\documentclass{article}
\usepackage{graphicx} 
\usepackage[a4paper, margin=2.5cm]{geometry}
\usepackage{amsthm} 
\usepackage{amssymb}
\newtheorem{theorem}{Theorem}

\newtheorem{proposition}{Proposition}

\newtheorem{definition}{Definition}

\usepackage{amsmath}
\usepackage{float}
\usepackage{subfigure}
\usepackage{tikz}
\usepackage{titlesec}
\titlelabel{\thetitle.\quad}
\usepackage{comment}

\title{Shuriken Graphs Arising from Clean Graphs of Rings and Their Properties Relative to Base Graphs}
\author{Felicia Servina Djuang$^{1}$, Indah Emilia Wijayanti$^{2}$, and Yeni Susanti$^{3}$\\
{\small $^{1,2,3}$Department of Mathematics, Universitas Gadjah Mada, Yogyakarta, Indonesia}\\
\small{$^{1}$feliciadjuang25@mail.ugm.ac.id, $^{2}$ind\_wijayanti@ugm.ac.id, $^{3}$yeni\_math@ugm.ac.id}}
\date{}

\begin{document}

\maketitle
\begin{abstract}
Let $R$ be a finite ring with identity. The idempotent graph $I(R)$ is the graph whose vertex set consists of the non-trivial idempotent elements of $R$, where two distinct vertices $x$ and $y$ are adjacent if and only if $xy = yx = 0$. The clean graph $Cl_2(R)$ is a graph whose vertices are of the form $(e, u)$, where $e$ is a nonzero idempotent element and $u$ is a unit of $R$. Two distinct vertices $(e,u)$ and $(f, v)$ are adjacent if and only if $ef = fe = 0$ or $uv = vu = 1$. The shuriken graph operation is an operation that arises from the structure of the clean graph and depends on the structure of the associated idempotent graph. In this paper, we study the graph obtained from the shuriken operation and examine how its properties depend on those of the base graph. In particular, we investigate several graph invariants, including the clique number, chromatic number, independence number, and domination number. Moreover, we analyze topological indices and characterize Eulerian and Hamiltonian properties of the resulting shuriken graphs in terms of the properties of the base graphs.\\
\textbf{Keywords:} clean graph, idempotent graph, operation graph, graph invariant, Zagreb indices. \\
\textbf{2020 Mathematics Subject Classification:} {05C25, 05C75, 05C76, 05C09, 05C30}
\end{abstract}

\section{Introduction}
Graph theory provides a useful framework for studying algebraic structures by representing algebraic relations in the form of graphs. In ring theory, many graphs are defined by taking algebraic elements as vertices and describing their interactions through edges. This approach allows algebraic problems to be studied using graph-theoretic methods.

Let $R$ be a finite ring with identity. One important example is the idempotent graph $I(R)$. Its vertices are the non-trivial idempotent elements of $R$, and two distinct vertices are adjacent if and only if they annihilate each other \cite{akbari}. This graph describes orthogonality relations among idempotents and gives information about the internal structure of the ring. Further research on idempotent graphs over matrix rings was conducted by \cite{patil}, who determined the structure of the idempotent graph over the ring $M_{2}(\mathbb{F})$, where $\mathbb{F}$ is a field. Another related construction is the clean graph, which comes from the notion of clean rings, where elements are written as sums of idempotents and units. By using both idempotents and units, clean graphs contain more structural information than idempotent graphs.

The clean graph $Cl_2(R)$ has vertices of the form $(e,u)$, where $e$ is a nonzero idempotent element and $u$ is a unit of $R$. Two vertices are adjacent when their idempotent parts are orthogonal or their unit parts are mutually inverses \cite{habibiyet}. Based on this definition, several structural and topological properties of the clean graph have been investigated in the literature. For instance, \cite{singhpatekar} studied the Wiener index of clean graphs over the ring $\mathbb{Z}_n$, where the analysis was restricted to cases in which $n$ has a prescribed number of prime factors. Due to its construction, the structure of $Cl_2(R)$ strongly depends on the structure of the associated idempotent graph. This observation leads naturally to graph operations that arise from clean graphs and depend on idempotent graphs.

Graph operations play a fundamental role in graph theory as they provide systematic ways to construct new graphs from existing ones and to analyze how structural properties and graph invariants behave under such transformations. In particular, unary graph operations, which associate a single input graph to a new graph, such as the line graph, graph complement, graph powers, and various local transformations, have proven to be essential tools in understanding both structural and algebraic aspects of graphs \cite{barik, gurski}. These operations allow researchers to investigate the stability and variation of graph invariants, including spectral parameters and width measures, under controlled transformations. Moreover, recent studies have emphasized the importance of defining and classifying new unary operations in order to enrich the theoretical framework of graph transformations and to capture more refined structural behaviors \cite{hasunama}. Consequently, the study of graph operations, especially unary operations, is not only mathematically significant but also urgent, as it provides a unifying language for constructing new graph classes and for advancing research in both pure graph theory and algebraically motivated graph models.

The shuriken graph operation is one such operation that appears in the study of clean graphs \cite{djuang}. It produces a graph whose structure depends on the structure of the base graph. In this paper, we study graphs obtained from the shuriken operation and focus on how their graph-theoretic properties relate to those of the base graph. For completeness, we recall the definition of the shuriken graph and present illustrative example as discussed in \cite{djuang, djuang2}.

\begin{definition}[\cite{djuang}]
Let $G=(V(G), E(G))$ be a graph and let $n, t$ be the positive integers such that $n-t$ is even. The $(t,n)$-shuriken graph of $G$, denoted by $Shu^t_n(G)$, is constructed from $G$ by first adding a new vertex ${z}$ and then creating $n$ copies of the resulting graph. Let $G'_i$ for $1 \leq i \leq n$ denote the $i$-th copy of $G$ after the addition of the new vertex. The vertex set and edge set of the graph $Shu^t_n(G)$ are given by:
    \begin{align*}
        V(Shu^t_n(G))=&\bigcup_{i=1}^n \{z_i, v_i: v \in V(G)\} \text{ and }\\
        E(Shu^t_n(G))= &\{u_iv_j : uv \in E(G), i,j \in \{1,2,\dots,n\}\} \\
        &\cup \{u_iv_i: u_i,v_i \in V(G'_i), u_i\neq v_i, i \in\{1,2,\dots,t\}\}\\ 
        &\cup \Bigg\{u_iv_{n+t+1-i}: u_i \in V(G'_i), v_{n+t+1-i} \in V(G'_{n+t+1-i}), i \in \left\{t+1,t+2,\dots,\frac{n+t}{2}\right\} \Bigg\}.
    \end{align*}
    Moreover, the order and size of the graph $Shu^t_n(G)$ are given by
    \begin{align*}
    |V(Shu^t_n(G))|=n(|V(G)|+1), 
    |E(Shu^t_n(G))|=\frac{1}{2}\left(n|V(G)|^2 + (2n-t)|V(G)|+n-t\right)+(n-1)|E(G)|.
\end{align*}
\end{definition}

The following example illustrates two shuriken graphs constructed from the graph $P_3$ and $K_2$. For $P_3$, the vertex set is $V(P_3)=\{a,b,c\}$ and the edge set is $E(P_3)=\{ab,bc\}$, with parameter values $t=2$ and $n=4$. For $K_2$, the vertex set is $V(K_2)=\{a,b\}$ and the edge set is $E(K_2)=\{ab\}$, with parameter values $t=8$ and $n=16$.
\begin{figure}[H]
    \begin{center}
        \resizebox{0.3\textwidth}{!}{\begin{tikzpicture}  
				[scale=.9,auto=center,roundnode/.style={circle,fill=blue!40}]
				\node[roundnode, fill=red!30] (a1) at (-2,6) {$a_1$};  
				\node[roundnode, fill=red!30] (a2) at (0,6)  {$b_1$};  
				\node[roundnode, fill=red!30] (a3) at (2,6)  {$c_1$};
				\node[roundnode] (a4) at (1,8) {$z_1$};
				\node[roundnode, fill=red!30] (a5) at (2,-2)  {$a_2$};
				\node[roundnode, fill=red!30] (a6) at (0,-2) {$b_2$}; 
                \node[roundnode, fill=red!30] (a7) at (-2,-2) {$c_2$};  
				\node[roundnode] (a8) at (-1,-4)  {$z_2$};  
				\node[roundnode, fill=red!30] (a9) at (4,4)  {$a_3$};
				\node[roundnode, fill=red!30] (a10) at (4,2) {$b_3$};
				\node[roundnode, fill=red!30] (a11) at (4,0)  {$c_3$};
				\node[roundnode] (a12) at (6,1) {$z_4$}; 
                \node[roundnode, fill=red!30] (a13) at (-4,0) {$a_4$};  
				\node[roundnode, fill=red!30] (a14) at (-4,2)  {$b_4$};  
				\node[roundnode, fill=red!30] (a15) at (-4,4)  {$c_4$};
				\node[roundnode] (a16) at (-6,3) {$z_3$};

				\draw[thick](a1) -- (a2);  
				\draw[thick](a2) -- (a3);
                \draw[thick](a1) -- (a4);  
				\draw[thick](a2) -- (a4);  
				\draw[thick](a3) -- (a4);%
                \draw[bend left=30] (a1) to (a3);
                \draw[thick](a2) -- (a5);  
				\draw[thick](a2) -- (a7);
                \draw[thick](a2) -- (a9);  
				\draw[thick](a2) -- (a11);
                \draw[thick](a2) -- (a13);
                \draw[thick](a2) -- (a15);
				\draw[thick](a5) -- (a6);
				\draw[thick](a6) -- (a7); 
                \draw[thick](a5) -- (a8);  
				\draw[thick](a6) -- (a8);  
				\draw[thick](a7) -- (a8);%
                \draw[bend left=30] (a5) to (a7);
                \draw[thick](a6) -- (a1);  
				\draw[thick](a6) -- (a3);
                \draw[thick](a6) -- (a9);  
				\draw[thick](a6) -- (a11);
                \draw[thick](a6) -- (a13);
                \draw[thick](a6) -- (a15);
                \draw[thick](a9) -- (a10);  
				\draw[thick](a10) -- (a11); 
                \draw[thick](a9) -- (a12);  
				\draw[thick](a10) -- (a12);  
				\draw[thick](a11) -- (a12);%
                \draw[thick](a10) -- (a5);  
				\draw[thick](a10) -- (a7);
                \draw[thick](a10) -- (a1);  
				\draw[thick](a10) -- (a3);
                \draw[thick](a10) -- (a13);
                \draw[thick](a10) -- (a15);
				\draw[thick](a13) -- (a14);
				\draw[thick](a14) -- (a15); 
                \draw[thick](a13) -- (a16);  
				\draw[thick](a14) -- (a16);  
				\draw[thick](a15) -- (a16);%
                \draw[thick](a14) -- (a5);  
				\draw[thick](a14) -- (a7);
                \draw[thick](a14) -- (a9);  
				\draw[thick](a14) -- (a11);
                \draw[thick](a14) -- (a1);
                \draw[thick](a14) -- (a3);
                \draw[thick](a9) -- (a13);
				\draw[thick](a9) -- (a14); 
                \draw[thick](a9) -- (a15);  
				\draw[thick](a10) -- (a13);  
				\draw[thick](a10) -- (a14);
                \draw[thick](a10) -- (a15);  
				\draw[thick](a11) -- (a13);
                \draw[thick](a11) -- (a14);  
				\draw[thick](a11) -- (a15);
                \draw[thick](a12) -- (a16);
			\end{tikzpicture}}
            \resizebox{0.35\textwidth}{!}{\begin{tikzpicture}  
				[scale=.9,auto=center,roundnode/.style={circle,fill=blue!40}]
				\node[roundnode] (a1) at (-0.99, 11.89) {$a_1$};  
				\node[roundnode] (a2) at (0.99, 11.89)  {$b_1$};  
				\node[roundnode] (a3) at (0, 13.67)  {$z_1$};
				\node[roundnode] (a4) at (7.69, 9.16) {$a_2$};
				\node[roundnode] (a5) at (9.07, 7.75)  {$b_2$};
				\node[roundnode] (a6) at (9.7, 9.76) {$z_2$}; 
               \node[roundnode] (a7) at (11.89, 0.99) {$a_3$};
				\node[roundnode] (a8) at (11.89, -0.99)  {$b_3$};
				\node[roundnode] (a9) at (13.67, -0.01) {$z_3$}; 
				\node[roundnode] (a10) at (9.1, -7.72) {$a_4$};
				\node[roundnode] (a11) at (7.74, -9.08)  {$b_4$};
				\node[roundnode] (a12) at (9.59, -9.58) {$z_4$}; 
                \node[roundnode] (a13) at (0.99, -11.87) {$a_5$};  
				\node[roundnode] (a14) at (-1, -11.86)  {$b_5$};  
				\node[roundnode] (a15) at (0, -13.65)  {$z_{5}$};
				\node[roundnode] (a16) at (-7.66, -9.14) {$a_6$};
				\node[roundnode] (a17) at (-9.16, -7.64)  {$b_6$};
				\node[roundnode] (a18) at (-9.78, -9.7) {$z_{6}$}; 
                \node[roundnode] (a19) at (-11.89, -0.99) {$a_7$};  
				\node[roundnode] (a20) at (-11.89, 0.99)  {$b_7$};  
				\node[roundnode] (a21) at (-13.67, 0)  {$z_{7}$};
				\node[roundnode] (a22) at (-9.15, 7.65) {$a_8$};
				\node[roundnode] (a23) at (-7.64, 9.16)  {$b_8$};
				\node[roundnode] (a24) at (-9.7, 9.73) {$z_8$}; 
                \node[roundnode] (a25) at (3.72, 11.33) {$a_{9}$};  
				\node[roundnode] (a26) at (5.63, 10.52)  {$b_{9}$};  
				\node[roundnode] (a27) at (5.41, 12.6)  {$z_{16}$};
				\node[roundnode] (a28) at (10.43, 5.8) {$a_{10}$};
				\node[roundnode] (a29) at (11.31, 3.8)  {$b_{10}$};
				\node[roundnode] (a30) at (12.65, 5.67) {$z_{15}$}; 
                \node[roundnode] (a31) at (11.39, -3.54) {$a_{11}$};  
				\node[roundnode] (a32) at (10.66, -5.35)  {$b_{11}$};  
				\node[roundnode] (a33) at (12.57, -5.04)  {$z_{14}$};
				\node[roundnode] (a34) at (5.78, -10.44) {$a_{12}$};
				\node[roundnode] (a35) at (3.72, -11.33)  {$b_{12}$};
				\node[roundnode] (a36) at (5.5, -12.68) {$z_{13}$}; 
                \node[roundnode] (a37) at (-3.66, -11.35) {$a_{16}$};  
				\node[roundnode] (a38) at (-5.71, -10.47)  {$b_{16}$};  
				\node[roundnode] (a39) at (-5.44, -12.51)  {$z_{9}$};
				\node[roundnode] (a40) at (-10.6, -5.48) {$a_{15}$};
				\node[roundnode] (a41) at (-11.39, -3.54)  {$b_{15}$};
				\node[roundnode] (a42) at (-12.74, -5.24) {$z_{10}$}; 
                \node[roundnode] (a43) at (-11.37, 3.59) {$a_{14}$};  
				\node[roundnode] (a44) at (-10.59, 5.49)  {$b_{14}$};  
				\node[roundnode] (a45) at (-12.59, 5.27)  {$z_{11}$};
				\node[roundnode] (a46) at (-5.58, 10.54) {$a_{13}$};
				\node[roundnode] (a47) at (-3.61, 11.37)  {$b_{13}$};
				\node[roundnode] (a48) at (-5.32, 12.65) {$z_{12}$};

				\draw (a1) -- (a2);  
				\draw (a1) -- (a5);  
				\draw (a1) -- (a8);
				\draw (a1) -- (a11); 
                \draw (a1) -- (a14);  
				\draw (a1) -- (a17);  
				\draw (a1) -- (a20);
				\draw (a1) -- (a23); 
                \draw (a1) -- (a26);  
				\draw (a1) -- (a29);  
				\draw (a1) -- (a32);
				\draw (a1) -- (a35); 
                \draw (a1) -- (a38);
				\draw (a1) -- (a41); 
                \draw (a1) -- (a44);  
				\draw (a1) -- (a47);  
                \draw (a4) -- (a2);  
				\draw (a4) -- (a5);  
				\draw (a4) -- (a8);
				\draw (a4) -- (a11); 
                \draw (a4) -- (a14);  
				\draw (a4) -- (a17);  
				\draw (a4) -- (a20);
				\draw (a4) -- (a23); 
                \draw (a4) -- (a26);  
				\draw (a4) -- (a29);  
				\draw (a4) -- (a32);
				\draw (a4) -- (a35); 
                \draw (a4) -- (a38);  
				\draw (a4) -- (a41);  
				\draw (a4) -- (a44);
				\draw (a4) -- (a47); 
                \draw (a7) -- (a2);  
				\draw (a7) -- (a5);  
				\draw (a7) -- (a8);
				\draw (a7) -- (a11); 
                \draw (a7) -- (a14);  
				\draw (a7) -- (a17);  
				\draw (a7) -- (a20);
				\draw (a7) -- (a23); 
                \draw (a7) -- (a26);  
				\draw (a7) -- (a29);  
				\draw (a7) -- (a32);
				\draw (a7) -- (a35); 
                \draw (a7) -- (a38);  
				\draw (a7) -- (a41);  
				\draw (a7) -- (a44);
				\draw (a7) -- (a47); 
                \draw (a10) -- (a2);  
				\draw (a10) -- (a5);  
				\draw (a10) -- (a8);
				\draw (a10) -- (a11); 
                \draw (a10) -- (a14);  
				\draw (a10) -- (a17);  
				\draw (a10) -- (a20);
				\draw (a10) -- (a23); 
                \draw (a10) -- (a26);  
				\draw (a10) -- (a29);  
				\draw (a10) -- (a32);
				\draw (a10) -- (a35);
                \draw (a10) -- (a38);  
				\draw (a10) -- (a41);  
				\draw (a10) -- (a44);
				\draw (a10) -- (a47);
                \draw (a13) -- (a2);  
				\draw (a13) -- (a5);  
				\draw (a13) -- (a8);
				\draw (a13) -- (a11); 
                \draw (a13) -- (a14);  
				\draw (a13) -- (a17);  
				\draw (a13) -- (a20);
				\draw (a13) -- (a23); 
                \draw (a13) -- (a26);  
				\draw (a13) -- (a29);  
				\draw (a13) -- (a32);
				\draw (a13) -- (a35); 
                \draw (a13) -- (a38);  
				\draw (a13) -- (a41);  
				\draw (a13) -- (a44);
				\draw (a13) -- (a47); 
                \draw (a16) -- (a2);  
				\draw (a16) -- (a5);  
				\draw (a16) -- (a8);
				\draw (a16) -- (a11); 
                \draw (a16) -- (a14);  
				\draw (a16) -- (a17);  
				\draw (a16) -- (a20);
				\draw (a16) -- (a23); 
                \draw (a16) -- (a26);  
				\draw (a16) -- (a29);  
				\draw (a16) -- (a32);
				\draw (a16) -- (a35); 
                \draw (a16) -- (a38);  
				\draw (a16) -- (a41);  
				\draw (a16) -- (a44);
				\draw (a16) -- (a47); 
                \draw (a19) -- (a2);  
				\draw (a19) -- (a5);  
				\draw (a19) -- (a8);
				\draw (a19) -- (a11); 
                \draw (a19) -- (a14);  
				\draw (a19) -- (a17);  
				\draw (a19) -- (a20);
				\draw (a19) -- (a23); 
                \draw (a19) -- (a26);  
				\draw (a19) -- (a29);  
				\draw (a19) -- (a32);
				\draw (a19) -- (a35); 
                \draw (a19) -- (a38);  
				\draw (a19) -- (a41);  
				\draw (a19) -- (a44);
				\draw (a19) -- (a47); 
                \draw (a22) -- (a2);  
				\draw (a22) -- (a5);  
				\draw (a22) -- (a8);
				\draw (a22) -- (a11); 
                \draw (a22) -- (a14);  
				\draw (a22) -- (a17);  
				\draw (a22) -- (a20);
				\draw (a22) -- (a23); 
                \draw (a22) -- (a26);  
				\draw (a22) -- (a29);  
				\draw (a22) -- (a32);
				\draw (a22) -- (a35); 
                \draw (a22) -- (a38);  
				\draw (a22) -- (a41);  
				\draw (a22) -- (a44);
				\draw (a22) -- (a47); 
                \draw (a25) -- (a2);  
				\draw (a25) -- (a5);  
				\draw (a25) -- (a8);
				\draw (a25) -- (a11); 
                \draw (a25) -- (a14);  
				\draw (a25) -- (a17);  
				\draw (a25) -- (a20);
				\draw (a25) -- (a23); 
                \draw (a25) -- (a26);  
				\draw (a25) -- (a29);  
				\draw (a25) -- (a32);
				\draw (a25) -- (a35); 
                \draw (a25) -- (a38);  
				\draw (a25) -- (a41);  
				\draw (a25) -- (a44);
				\draw (a25) -- (a47); 
                \draw (a28) -- (a2);  
				\draw (a28) -- (a5);  
				\draw (a28) -- (a8);
				\draw (a28) -- (a11); 
                \draw (a28) -- (a14);  
				\draw (a28) -- (a17);  
				\draw (a28) -- (a20);
				\draw (a28) -- (a23); 
                \draw (a28) -- (a26);  
				\draw (a28) -- (a29);  
				\draw (a28) -- (a32);
				\draw (a28) -- (a35); 
                \draw (a28) -- (a38);  
				\draw (a28) -- (a41);  
				\draw (a28) -- (a44);
				\draw (a28) -- (a47); 
                \draw (a31) -- (a2);  
				\draw (a31) -- (a5);  
				\draw (a31) -- (a8);
				\draw (a31) -- (a11); 
                \draw (a31) -- (a14);  
				\draw (a31) -- (a17);  
				\draw (a31) -- (a20);
				\draw (a31) -- (a23); 
                \draw (a31) -- (a26);  
				\draw (a31) -- (a29);  
				\draw (a31) -- (a32);
				\draw (a31) -- (a35); 
                \draw (a31) -- (a38);  
				\draw (a31) -- (a41);  
				\draw (a31) -- (a44);
				\draw (a31) -- (a47); 
                \draw (a34) -- (a2);  
				\draw (a34) -- (a5);  
				\draw (a34) -- (a8);
				\draw (a34) -- (a11); 
                \draw (a34) -- (a14);  
				\draw (a34) -- (a17);  
				\draw (a34) -- (a20);
				\draw (a34) -- (a23); 
                \draw (a34) -- (a26);  
				\draw (a34) -- (a29);  
				\draw (a34) -- (a32);
				\draw (a34) -- (a35); 
                \draw (a34) -- (a38);  
				\draw (a34) -- (a41);  
				\draw (a34) -- (a44);
				\draw (a34) -- (a47); 
                \draw (a37) -- (a2);  
				\draw (a37) -- (a5);  
				\draw (a37) -- (a8);
				\draw (a37) -- (a11); 
                \draw (a37) -- (a14);  
				\draw (a37) -- (a17);  
				\draw (a37) -- (a20);
				\draw (a37) -- (a23); 
                \draw (a37) -- (a26);  
				\draw (a37) -- (a29);  
				\draw (a37) -- (a32);
				\draw (a37) -- (a35); 
                \draw (a37) -- (a38);  
				\draw (a37) -- (a41);  
				\draw (a37) -- (a44);
				\draw (a37) -- (a47);
                \draw (a40) -- (a2);  
				\draw (a40) -- (a5);  
				\draw (a40) -- (a8);
				\draw (a40) -- (a11); 
                \draw (a40) -- (a14);  
				\draw (a40) -- (a17);  
				\draw (a40) -- (a20);
				\draw (a40) -- (a23); 
                \draw (a40) -- (a26);  
				\draw (a40) -- (a29);  
				\draw (a40) -- (a32);
				\draw (a40) -- (a35); 
                \draw (a40) -- (a38);  
				\draw (a40) -- (a41);  
				\draw (a40) -- (a44);
				\draw (a40) -- (a47); 
                \draw (a43) -- (a2);  
				\draw (a43) -- (a5);  
				\draw (a43) -- (a8);
				\draw (a43) -- (a11); 
                \draw (a43) -- (a14);  
				\draw (a43) -- (a17);  
				\draw (a43) -- (a20);
				\draw (a43) -- (a23); 
                \draw (a43) -- (a26);  
				\draw (a43) -- (a29);  
				\draw (a43) -- (a32);
				\draw (a43) -- (a35); 
                \draw (a43) -- (a38);  
				\draw (a43) -- (a41);  
				\draw (a43) -- (a44);
				\draw (a43) -- (a47); 
                \draw (a46) -- (a2);  
				\draw (a46) -- (a5);  
				\draw (a46) -- (a8);
				\draw (a46) -- (a11); 
                \draw (a46) -- (a14);  
				\draw (a46) -- (a17);  
				\draw (a46) -- (a20);
				\draw (a46) -- (a23); 
                \draw (a46) -- (a26);  
				\draw (a46) -- (a29);  
				\draw (a46) -- (a32);
				\draw (a46) -- (a35); 
                \draw (a46) -- (a38);  
				\draw (a46) -- (a41);  
				\draw (a46) -- (a44);
				\draw (a46) -- (a47); 
                \draw (a1) -- (a3);  
				\draw (a2) -- (a3);  
				\draw (a4) -- (a6);
				\draw (a5) -- (a6);
                \draw (a7) -- (a9);  
				\draw (a8) -- (a9);  
				\draw (a10) -- (a12);
				\draw (a11) -- (a12);
                \draw (a13) -- (a15);  
				\draw (a14) -- (a15);  
				\draw (a16) -- (a18);
				\draw (a17) -- (a18);
                \draw (a19) -- (a21);  
				\draw (a20) -- (a21);  
				\draw (a22) -- (a24);
				\draw (a23) -- (a24);
                \draw (a25) -- (a27);  
				\draw (a26) -- (a27);  
				\draw (a28) -- (a30);
				\draw (a29) -- (a30);
                \draw (a31) -- (a33);  
				\draw (a32) -- (a33);  
				\draw (a34) -- (a36);
				\draw (a35) -- (a36);
                \draw (a37) -- (a39);
				\draw (a38) -- (a39);
                \draw (a40) -- (a42);  
				\draw (a41) -- (a42);  
				\draw (a43) -- (a45);
				\draw (a44) -- (a45);
                \draw (a46) -- (a48);
				\draw (a47) -- (a48);
                \draw (a25) -- (a37);  
				\draw (a26) -- (a38);  
				\draw (a27) -- (a39);
                \draw (a28) -- (a40);  
				\draw (a29) -- (a41);  
				\draw (a30) -- (a42);
                \draw (a31) -- (a43);  
				\draw (a32) -- (a44);  
				\draw (a33) -- (a45);
                \draw (a34) -- (a46);  
				\draw (a35) -- (a47);  
				\draw (a36) -- (a48);
			\end{tikzpicture}}
            \caption{Graphs $Shu^2_{4}(P_3)$ and $Shu^8_{16}(K_2)$}\label{grafShu24P3}
		\end{center}
\end{figure}
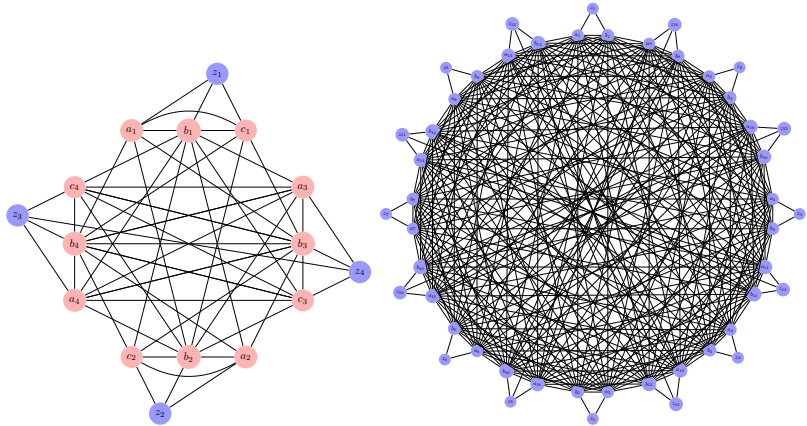

A fundamental structural property of shuriken graphs concerns their connectivity. In particular, previous studies have established a necessary and sufficient condition for a shuriken graph to be connected, showing that connectivity holds whenever the base graph is not a null graph \cite{djuang}. Under this condition, the resulting shuriken graph is connected, which allows the investigation of several graph parameters that are defined for connected graphs. This observation provides a natural framework for studying classical graph invariants within the class of shuriken graphs.

Furthermore, topological indices associated with clean graphs over rings have also studied through the lens of the shuriken graph structure. In \cite{djuang3}, several topological indices, including the first and second Zagreb indices as well as the Randić index, were investigated for clean graphs over the ring $\mathbb{Z}_n$ by exploiting the properties of the corresponding shuriken graphs. These results further demonstrate the usefulness of the shuriken operation as a tool for connecting algebraic graph constructions with well-known graph-theoretic invariants.

Before proceeding, we recall several basic graph-theoretic concepts that are used throughout this paper. Let $G=(V(G),E(G))$ be a graph, where $V(G)$ and $E(G)$ denote the vertex set and the edge set of graph $G$, respectively. A graph $G$ is said to be \textbf{connected} if there exists a path between every pair of distinct vertices, and \textbf{complete} if every pair of distinct vertices is adjacent. The graph $G$ is called a \textbf{null graph} if $E(G) = \emptyset$. For a vertex $v \in V(G)$, the degree of $v$, denoted by $d_G(v)$, is the number of edges incident with $v$ on graph $G$. The \textbf{clique number} of $G$, denoted by $\omega(G)$, is the maximum cardinality of a subset of vertices that induces a complete subgraph, while the \textbf{chromatic number} $\chi(G)$ is the minimum number of colors required to color the vertices of $G$ such that adjacent vertices receive different colors. The \textbf{independence number} $\alpha(G)$ is the maximum cardinality of a subset of vertices that induces a null subgraph. A \textbf{dominating set} of $G$ is a subset $D \subseteq V(G)$ such that every vertex in $V(G) \setminus D$ is adjacent to at least one vertex in $D$, and the \textbf{domination number} $\gamma(G)$ is the minimum cardinality of such a set. Furthermore, the first and second Zagreb indices of $G$, denoted by $M_1(G)$ and $M_2(G)$, are defined by 
\begin{align*}
		M_1(G) = \sum_{v \in V(G)} \left( \deg_G(v) \right)^2, \quad M_2(G) = \sum_{uv \in E(G)} \left( \deg_G(u) \deg_G(v) \right).
	\end{align*}

Subsequently, we discuss graph operations that will be used in the results of this study. Two basic operations that are frequently used are the \textbf{disjoint union} and the \textbf{join union} of graphs. Given two graphs $G$ and $H$, the disjoint union, denoted by $G \cup H$, is the graph whose vertex set and edge set are the unions of those of $G$ and $H$, with no additional edges between the two graphs. In contrast, the join of $G$ and $H$, denoted by $G + H$, is obtained from $G \cup H$ by adding all possible edges between every vertex of $G$ and every vertex of $H$. These operations are simple to define, yet they play an important role in graph theory, since many graph invariants and structural properties can be studied through how they behave under disjoint unions and joins.

Moreover, a graph $G$ is called \textbf{Hamiltonian} if it contains a cycle that visits every vertex exactly once, and \textbf{Eulerian} if it contains a closed trail that uses every edge exactly once. The characterization of Eulerian graphs was developed progressively in the literature. While Leonhard Euler initiated the study of traversals using each edge exactly once, a complete necessary and sufficient condition for a connected graph to be Eulerian was established later. A rigorous formulation of this characterization is commonly attributed to Karl Hierholzer, who provided a constructive proof for the existence of an Eulerian circuit in connected graphs whose vertices all have even degree \cite{hierholzer}. For additional background on graph theory and relevant terminology, we refer the reader to \cite{wilson}.

Research on algebraic graphs has attracted attention due to its potential applications in coding theory. In particular, \cite{fish} constructed linear codes using the incidence matrix of the line graph of the Hamming graph, demonstrating how combinatorial structures derived from graphs can be used to generate error correcting codes. Motivated by this approach, \cite{jain} extended the idea by constructing linear codes from the incidence matrix of the unit graph over $\mathbb{Z}_n$, where different constructions are obtained by classifying cases according to the number of prime factors of $n$. This result illustrates the relevance of studying structural and invariant properties of algebraic graphs, as such properties often play an essential role in determining the characteristics of the associated codes.

In particular, we study several classical graph invariants of shuriken graphs, including the clique number, the chromatic number, the independence number, and the domination number. We also examine several topological indices and study Eulerian and Hamiltonian properties. Our results describe how these properties depend on the properties of the base graph and provide a clearer picture of graph constructions that arise from finite rings. Moreover, as the shuriken graph operation arises naturally from algebraic graphs, namely clean graphs and idempotent graphs associated with finite rings, this perspective not only provides a clearer picture of graph constructions induced by ring-theoretic structures but also enriches the existing collections of graph operations by introducing constructions motivated by algebraic considerations.


\section{Result and Discussion}

In this part, we focus on the relationship between the clique number of a graph and the clique number of its associated shuriken graph. Since the shuriken construction is defined through a systematic expansion of the base graph, it is natural to expect that complete subgraphs in the base graph play an important role in determining the size of complete subgraphs in the resulting shuriken graph. Understanding how cliques are preserved, enlarged, or restricted under the shuriken operation provides useful insight into the structural behavior of this construction. 

The results presented below describe explicit connections between the clique number of a graph and that of its shuriken graph, highlighting how this invariant is affected by the parameters of the shuriken operation.
\begin{theorem}
    Given graph $G=(V(G),E(G))$ and shuriken graph $Shu^t_n(G)$ of $G$, for some $t,n \in \mathbb{Z}^+$, where $0\leq n-t$ is even. We have
    \begin{align*}
        \omega(Shu^t_n(G))=\begin{cases}
            |V(G)|+1, &\text{if } n-t=0\\
            \max\{|V(G)|+1,2\omega(G)\}, &\text{if }n-t>0.
        \end{cases}
    \end{align*}
\end{theorem}
\begin{proof}
We consider two cases as follows:
    \begin{enumerate}
        \item Case $n-t=0 \iff n=t$. Let $X=V(G_1')$. By the definition of operation of shuriken graph, $X$ is a clique in $Shu^t_n(G)$. Moreover, $|X|=|V(G_1')|=|V(G)|+1$. Assume that there exists a clique $Y \subseteq V(Shu^t_n(G))$, such that $|Y| > |X|$. It means that $Y$ contains some vertices from two different copies. Let $a_i$ and $b_j$ be the vertices from $G'_i$ and $G'_j$, respectively, where $a,b \in V(G) \cup \{z\}$ and $i \neq j$. If $a = b$, then $a_i$ and $b_j$ are not adjacent. This contradicts the fact that $Y$ is a clique. Thus, $a \neq b$. Using the Pigeonhole Principle, we have $|Y| \leq |V(G)|+1$. This contradicts the fact that $|Y| > |X|$. \label{poin1}
        \item Case $n-t>0$. Assume that $A \subseteq V(G)$ is a clique of the graph $G$. WLOG, let $X=V(G_1')$ and $Y=\{a_{t+1},a_n: a \in A\}$. Clearly, $X$ is a clique in $Shu^t_n(G)$, where $|X|=|V(G_1')|=|V(G)|+1$. On the other hand, for any $x,y \in Y$, where $x \neq y$, we have $\{x,y\}=\{a_{i},b_i\}$ or $\{x,y\}=\{a_i,c_j\}$ for $i,j \in \{t+1,n\}$ where $i \neq j$, $a,b,c \in A$, and $a \neq b$. Since $A$ is a clique of $G$ and $a \neq b$, then $a_i,b_i$ are adjacent. Since $i,j \in \{t+1,n\}$ and $i \neq j$, using the definition of the shuriken graph, $a_i$ and $b_j$ are adjacent. Hence, $Y$ is a clique in $Shu^t_n(G)$, where $|Y|=2\omega(G)$.\\
        Assume that there is a clique $Z \subseteq V(Shu^t_n(G))$, where $|Z| > \max\{|X|,|Y|\}$. Necessarily, $Z$ contains at least two vertices from different copies. If the copies are unpaired (the paired copies are $G'_i$ and $G'_{n+t+1-i}$, where $t+1 \leq i \leq n$), using the Pigeonhole Principle as in \ref{poin1}, we obtain $|Z| \leq |V(G)|+1$. This contradicts the fact that $|Z| > \max\{|X|,|Y|\}$. Thus, there exists a paired copies in $Z$, WLOG let $(t+1)^{th}$ copy and $(n)^{th}$ copy. Since $|Z| > |Y|=2 \omega(G)$, we obtain $Y \subseteq Z$ and there exists a vertex not from $A$, namely $v$, such that $v_{t}$ or $v_{n}$ in $Z$, or there is $u_s \in V(Shu^t_n(G))$, where $s \notin \{t+1,n\}$, such that $u_s \in Z$. For the first case, since $Z$ is a clique, for all $a_{t+1},a_n \in Z$, we have $a_{t+1}v_{t+1} \in E(Shu^t_n(G))$ and $a_nv_n \in E(Shu^t_n(G))$. Hence $av \in E(G)$ for all $a \in A$. It means that $v$ is in the clique of $G$, a contradiction. For the second case, if $u \in A$, there exists $u_t \in Z$ such that $u_su_t \notin E(Shu^t_n(G))$, which leads to a contradiction with the fact that $Z$ is a clique. Hence, $u \notin A$. However, for all $a_{t+1},a_n \in Z$, we have $a_{t+1}u_s \in E(Shu^t_n(G))$. Using a similar argument, we know $au \in E(G)$, thus $u$ is in the clique of $G$, a contradiction. \\
        Consequently, for all clique $Z \subseteq V(Shu^t_n(G))$, $|Z| \leq \max\{|X|,|Y|\}$. Hence $$\omega(Shu^t_n(G))=\max\{|V(G)|+1, 2\omega(G)\}.$$
    \end{enumerate}
\end{proof}

Later, we study the chromatic number of shuriken graphs in relation to the chromatic number of their base graphs. Determining the exact value of the chromatic number for shuriken graphs appears to be a challenging problem in general. For this reason, we focus on establishing meaningful upper and lower bounds and on describing vertex coloring schemes induced by proper colorings of the base graph. Due to the structured nature of the shuriken construction, colorings of the base graph can often be extended in a controlled way to the shuriken graph. The results presented in this subsection clarify how the parameters of the shuriken operation constrain possible colorings and provide insight into the chromatic behavior of shuriken graphs.
\begin{proposition}\label{chromatic}
    Given graph $G=(V(G),E(G))$ and shuriken graph $Shu^t_n(G)$ of $G$, for some $t,n \in \mathbb{Z}^+$, where $0\leq n-t$ is even. If $f$ represents a vertex coloring in $G$ such that every edge in $G$ connects two vertices of different colors used to achieve its chromatic number, we have
\begin{align*}
    \chi(Shu^t_n(G))&=|V(G)|+1, \quad \text{if } n-t=0\\
    \chi(Shu^t_n(G))&\geq \max\{|V(G)|+1, 2\chi(G)+ \varphi\}, \quad \text{if } n-t>0
\end{align*}
where $$\varphi = \sum_{\substack{1 \leq k \leq \chi(G)\\ |A_k|>2}}(|A_k|-2)$$
and $A_k=\{x \in f^{-1}(k): (\forall y \in V(G)\setminus f^{-1}(k)), xy \in E(G)\}$.
\end{proposition}
\begin{proof}
    We have $Shu^t_n(G)$, for some $t,n \in \mathbb{Z}^+$, where $0\leq n-t$ is even. We have two cases below.
    \begin{enumerate}
        \item Case $n-t=0$. We know that $\chi(Shu^t_n(G))\geq \omega(Shu^t_n(G))=|V(G)|+1$. Thus, $\chi(Shu^t_n(G))\geq |V(G)|+1$. Partitions of the set $V(Shu^t_n(G))$ are formed as follows: $X_u=\{u_i: 1 \leq i\leq n\}$ for $u \in V(G) \cup \{z\}$. Consequently, for any $u_i,u_j \in X_u$, for $u \in V(G) \cup \{z\}$, $u_iu_j \notin E(Shu^t_n(G))$ since $n=t$ and $G$ is a simple graph. 
        \item Case $n-t>0$. We know that $\chi(Shu^t_n(G))\geq \omega(Shu^t_n(G))=\max\{|V(G)|+1,2\omega(G)\}$. If the chromatic number of the graph $G$ is $\chi(G)$, observe that the subgraph of $Shu^t_n(G)$ induced by $\{x_{t+1}, x_n: x \in V(G)\}$ is $G + G$. Since $\chi(G+G)=2\chi(G)$, it follows that $\chi(Shu^t_n(G)) \geq 2 \chi(G)$. On the other hand, since $ \chi(G) \geq \omega(G)$, we have $\chi(Shu^t_n(G))\geq \max\{|V(G)|+1,2\chi(G)\}$. Let $f$ represent an optimum vertex coloring in $G$. For each $k$, where $1 \leq k \leq \chi(G)$, if $|A_k|>2$, let $A_k=\{x_1,x_2,\dots,x_{|A_k|}\}$, and consider the vertices $x_{1_1},x_{2_1},\dots,x_{|A_k|_1}$ in $G'_1$. Since these vertices are adjacent to one another, their colors must all be distinct. Furthermore, because these vertices are adjacent to $y_i$ for every $y \in V(G) \setminus f^{-1}(k)$ and $1 \leq i \leq n$, and since $\chi(Shu^t_n(G))\geq 2\chi(G)$, there remain two color choices available for the vertices $x_{1_1},x_{2_1},\dots,x_{|A_k|_1}$. Consequently, $$\chi(Shu^t_n(G))\geq 2 \chi(G)+ \sum_{\substack{1 \leq k \leq \chi(G)\\ |A_k|>2}}(|A_k|-2)=2 \chi(G)+ \varphi.$$
    \end{enumerate}
\end{proof}

A proper vertex coloring of the base graph can be used to define a vertex coloring of its shuriken graph. By assigning colors to vertices in each copy according to the coloring of the base graph and making suitable adjustments to account for additional adjacencies, one obtains a proper coloring of the shuriken graph. The following theorem makes this construction precise.

\begin{theorem}
    Given graph $G=(V(G),E(G))$ and shuriken graph $Shu^t_n(G)$ of $G$, for some $t,n \in \mathbb{Z}^+$, where $0\leq n-t$ is even. Let $f: V(G) \to  \{1,2,\dots,\chi(G)\}$ be a coloring of the graph $G$. Then $g: V(Shu^t_n(G)) \to \mathbb{N}$, where 
    \begin{enumerate}
    \item for each $i$, where $t+1 \leq i \leq \tfrac{n+t}{2}$ and for every $x \in V(G)$, $g(x_i)=f(x)$ and $g(z_i)=1$,
    
    \item for each $i$, where $\tfrac{n+t}{2}+1 \leq i \leq n$ and for every $x \in V(G)$, $g(x_i)=f(x)+\chi(G)$ and $g(z_i)=\chi(G)+1$,
    
    \item for each $i$, where $1 \leq i \leq t$, proceed as follows:
    \begin{enumerate}
        \item[a. ] for each vertex 
        $x \in \{x \in V(G): f^{-1}(\{f(x)\})=\{x\}\}$ (Type 1), $g(x_i)=f(x)$,
        
        \item[b. ] for each pair
        $\{x,y\} \in \{\{x,y\} \subseteq V(G): x \neq y,\ f^{-1}(\{f(x)\})=\{x,y\}\}$ (Type 2), $g(x_i)=f(x)$ and $g(y_i)=f(x)+\chi(G)$,
        
        \item[c. ] for each
        $\{a_1,a_2,\dots, a_k \}\in \{\{a_1,a_2,\dots, a_k \} \subseteq V(G): f^{-1}(\{f(a_1)\})=\{a_1,a_2,\dots,a_k\}\}$, where $k \geq 3$. Suppose that
        $\{a_1,a_2,\dots,a_k\}$
        is the set of vertices ordered according to their degrees in $G$ (from the largest to the smallest), then
        $g({a_1}_i)=f(x)$ and $g({a_2}_i)=f(x)+\chi(G)$,
        
        for the vertex $a_3$, the following cases apply:
        \begin{enumerate}
            \item if there exists a Type 1 vertex $b \in V(G)$ such that $a_3$ is not adjacent to $b$ in $G$, then
            $g({a_3}_i)=f(b)+\chi(G)$,
            
            \item if there is no Type 1 vertex in $V(G)$ that is not adjacent to $a_3$ in $G$, then
            $g({a_3}_i)=2\chi(G)+1$,
        \end{enumerate}
        
        for each vertex $a_j$ with $3 < j \leq k$, the following rules apply:
        \begin{enumerate}
            \item if there exists a Type 1 vertex $c \in V(G)$ that has not been used at any previous index and $a_j$ is not adjacent to $c$ in $G$, then
            $g({a_j}_i)=f(c)+\chi(G)$,
            
            \item if there is no unused Type 1 vertex in $V(G)$ that is not adjacent to $a_j$ in $G$, then
            $$g({a_j}_i)= \max\left\{2 \chi(G)+1, \max_{l<j}\{g({a_l}_i)\}+1 \right\},$$
        \end{enumerate}
        \item[d. ] $g(z_i)=\min \{y \in \mathbb{N} \setminus \{g(x_{i}): x \in V(G)\}\}$,
    \end{enumerate}
\end{enumerate}
is a coloring of the graph $Shu^t_n(G)$.
\end{theorem}
\begin{proof}
    We examine all edges of the shuriken graph $Shu^t_n(G)$.
    For any $i$, where $t+1 \leq i \leq \tfrac{n+t}{2}$, we have $\{g(x_i): x \in V(G) \cup \{z\}\} \cap \{g(x_{n+t+1-i}): x \in V(G) \cup \{z\}\} = \emptyset$. Consequently, for any $x,y \in V(G) \cup \{z\}$, $g(x_i) \neq g(y_{n+t+1-i})$. Now, assume that $g(u_j)=g(v_j)$ for some $u,v \in V(G) \cup \{z\}$ and $1 \leq j \leq t$, where $u \neq v$. If $f(u)=f(v)+\chi(G)$ or $f(v)=f(u)+\chi(G)$, then it follows that $\chi(G)=f(u)-f(v)$ or $\chi(G)=f(v)-f(u)$, which leads to a contradiction with the fact that $\chi(G)=\max\{f(x):x \in V(G)\}$. It must be $f(u)=f(v)$, such that $|f^{-1}(\{f(u)\})| \geq 2$, where $u,v \in f^{-1}(\{f(u)\})$. Thus, $g(u_j)=f(u)$ and $g(v_j)=f(u)+\chi(G)$, or $g(u_j)=f(v)+\chi(G)$ and $g(v_j)=f(v)$, or $g(u_j),g(v_j)> 2\chi(G)$, in which case $g(u_j)$ and $g(v_j)$ are necessarily distinct. Hence, a contradiction also follows. Let $c,d \in V(G)$ such that $cd \in E(G)$, and $1 \leq i,j \leq n$. Consider the following cases:
    \begin{enumerate} 
    \item If \( t+1 \leq i,j \leq \tfrac{n+t}{2} \) or \( \tfrac{n+t}{2}+1 \leq i,j \leq n \), then \( f(c) \neq f(d) \iff g(c_i) \neq g(d_j) \). \item If \( t+1 \leq i \leq \tfrac{n+t}{2} \) and \( \tfrac{n+t}{2}+1 \leq j \leq n \), or vice versa, then \( g(c_i) \in \{1,2,\dots, \chi(G)\} \) and \( g(d_j) \in \{\chi(G)+1,\chi(G)+2,\dots, 2\chi(G)\} \), or vice versa. Consequently, \( g(c_i) \neq g(d_j) \). 
    \item If \( 1 \leq i \leq t \) or $1 \leq j \leq t$ with \( g(c_i) = g(d_j) \), then it must be that \( f(c) = f(d) \), which leads to a contradiction since \( cd \in E(G) \). Hence, $g(c_i) \neq g(d_j)$. \end{enumerate}
    Therefore, every pair of vertices that forms an edge in $Shu^t_n(G)$ is assigned distinct colors under the coloring $g$.
\end{proof}

We examine the independence number of shuriken graphs in relation to the independence number of their base graphs. Since the shuriken construction introduces new adjacency relations while preserving copies of the base graph, it is natural to study how independent sets in the base graph can be extended or modified in the shuriken graph. 
\begin{theorem}
    Given graph $G=(V(G),E(G))$ and shuriken graph $Shu^t_n(G)$ of $G$, for some $t,n \in \mathbb{Z}^+$, where $0\leq n-t$ is even. We have
    \begin{align*}
        \alpha(Shu^t_n(G))=t+\frac{n-t}{2}(\alpha(G)+1).
    \end{align*}
\end{theorem}
\begin{proof}
    Since for every $i$, where $1 \leq i \leq t$, the subgraph of $Shu^t_n(G)$ induced by $A_i=\{x_i: x \in V(G) \cup \{z\}\}$ is a complete graph, it follows that $\alpha(\langle A_i \rangle_{Shu^t_n(G)})= 1$. For each $i$, where $t+1 \leq i \leq\tfrac{n+t}{2}$, the subgraph of $Shu^t_n(G)$ induced by $B_i=\{x_i, x_{n+t+1-i}: x \in V(G)\cup \{z\}\}$ is given by $\langle B_i \rangle_{Shu^t_n(G)} = (G \cup K_1) + (G \cup K_1)$, the join of two copies of $G \cup K_1$. Hence, $\alpha(\langle B_i \rangle_{Shu^t_n(G)})=\alpha(G)+1$. Since the graph $$H=\bigcup_{i=1}^t \langle A_i \rangle_{Shu^t_n(G)} \cup \bigcup_{i=t+1}^{\tfrac{n+t}{2}} \langle B_i \rangle_{Shu^t_n(G)}$$ is a spanning subgraph of $Shu^t_n(G)$, we obtain $\alpha(Shu^t_n(G)) \leq \alpha(H)=t+\frac{n-t}{2}(\alpha(G)+1)$.\\
        Let $X \subseteq V(G)$ be an independent set of $G$, with $|X|=\alpha(G)$. Consider the set $$Z = \{z_i: 1 \leq i \leq \tfrac{n+t}{2} \} \cup \{x_j: x \in X, t+1 \leq j \leq \tfrac{n+t}{2}\} \subseteq V(Shu^t_n(G)).$$ Note that
        \begin{align*}
            N_{Shu^t_n(G)}(z_i)=\begin{cases}
                \{x_i: x \in V(G)\}, &\text{if } 1 \leq i \leq t\\
                \{z_{n+t+1-i}, x_{n+t+1-i}: x \in V(G)\}, &\text{if } t+1 \leq i \leq \tfrac{n+t}{2}.
            \end{cases}
        \end{align*}
        It follows that 
        \begin{align*}
            \bigcup_{i=1}^{\tfrac{n+t}{2}} N_{Shu^t_n(G)}(z_i) = \{x_i: x \in V(G), 1 \leq i \leq t \text{ atau } \tfrac{n+t}{2}+1 \leq i \leq n\} \cup \{z_i : \tfrac{n+t}{2}+1 \leq i \leq n\}.
        \end{align*}
        Consequently, the set $\{z_i: 1 \leq i \leq \tfrac{n+t}{2}\}$ is an independent set in the graph $Shu^t_n(G)$, and each vertex in $\{z_i: 1 \leq i \leq \tfrac{n+t}{2}\}$ is not adjacent to every vertex in $\{x_j: x \in X, t+1 \leq j \leq \tfrac{n+t}{2}\}$. Assume that there exist $x_i, y_j \in \{x_j: x \in X, t+1 \leq j \leq \tfrac{n+t}{2}\}$ such that $x_iy_j \in E(Shu^t_n(G))$. This implies that either $xy \in E(G)$ or $j=n+t+1-i$. Since $t+1 \leq i,j \leq \tfrac{n+t}{2}$, the equality $j=n+t+1-i$ cannot occur. If $xy \in E(G)$, while $x, y \in X$, this contradicts the fact that $X$ is an independent set in $G$. Hence, $\{x_j: x \in X, t+1 \leq j \leq \tfrac{n+t}{2}\}$ is also an independent set in $Shu^t_n(G)$. It follows that $Z$ is an independent set in $Shu^t_n$ with $|Z|=\tfrac{n+t}{2} + \tfrac{n-t}{2}\alpha(G)=t+\tfrac{n-t}{2} + \tfrac{n-t}{2}\alpha(G)=t+\tfrac{n-t}{2}(\alpha(G)+1)$. Therefore, $\alpha(G) \geq t+\tfrac{n-t}{2}(\alpha(G)+1)$, and consequently $\alpha(G) = t+\tfrac{n-t}{2}(\alpha(G)+1)$.
\end{proof}

We study the domination number of shuriken graphs in terms of the domination number of their base graphs. The structured nature of the shuriken construction allows dominating sets of the base graph to be extended to dominating sets of the shuriken graph, leading to bounds that relate the two domination numbers.
\begin{theorem}
    Given graph $G=(V(G),E(G))$ and shuriken graph $Shu^t_n(G)$ of $G$, for some $t,n \in \mathbb{Z}^+$, where $0\leq n-t$ is even. If $\gamma(G) \leq t$, then
    \begin{align*}
        \gamma(Shu^t_n(G))= \frac{n+t}{2}.
    \end{align*}
    If $\gamma(G) \geq t$, then
    $$\frac{n+t}{2} \leq \gamma(Shu^t_n(G)) \leq \gamma(G)+\frac{n-t}{2}.$$
\end{theorem}
\begin{proof}
Assume that $\gamma(Shu^t_n(G)) < \tfrac{n+t}{2}$, and let $X \subseteq V(Shu^t_n(G))$ be a dominating set. For each $i$, where $1 \leq i \leq \tfrac{n+t}{2}$, define $$A_i=\begin{cases}
            \{x_i: x \in V(G) \cup \{z\}\}, &\text{if } 1 \leq i \leq t\\
            \{x_i, x_{n+t+1-i} :  x \in V(G) \cup \{z\}\}, &\text{if } t+1 \leq i \leq \tfrac{n+t}{2}.
        \end{cases}$$
        Then $$\bigcup_{i=1}^{\tfrac{n+t}{2}} A_i = V(G) \text{ and for all } 1 \leq i,j \leq \tfrac{n+t}{2}, \text{ if }  i\neq j, \text{ then }A_i \cap A_j = \emptyset.$$ Since $\gamma(Shu^t_n(G)) < \tfrac{n+t}{2}$, there exists $1 \leq k \leq \tfrac{n+t}{2}$ such that $X \cap A_k = \emptyset$. Moreover, 
        \begin{align*}
            N_{Shu^t_n(G)}(z_i)=\begin{cases}
                \{x_i: x \in V(G)\}, &\text{if } 1 \leq i \leq t\\
                \{z_{n+t+1-i}, x_{n+t+1-i}: x \in V(G)\}, &\text{if } t+1 \leq i \leq n.
            \end{cases}
        \end{align*} 
        It follows that $N_{Shu^t_n(G)}(z_i) \in A_i$ for every $1 \leq i \leq t$, and $N_{Shu^t_n(G)}(z_i) \cup N_{Shu^t_n(G)}(z_{n+t+1-i}) \in A_i$ for each $t+1 \leq i \leq \tfrac{n+t}{2}$. Hence, the vertex $z_k \in A_k$, $z_k \notin X$, and $z_k$ has no neighbors in $X$ because $N_{Shu^t_n(G)}(z_k) \in A_k$ and $X \cap A_k = \emptyset$. This contradicts the assumption that $X$ is a dominating set. Consequently, $\gamma(Shu^t_n(G)) \geq \tfrac{n+t}{2}$. Let $X$ be a dominating set of $G$ with $|X|=\gamma(G)$. Write $X=\{a_1,a_2,\dots, a_{\gamma(G)}\} \subseteq V(G)$.
    \begin{enumerate}
            \item Case $\gamma(G) \leq t$. Define $$Z=\left\{a_{1_1}, a_{2_2}, \dots, a_{\gamma(G)_{\gamma(G)}}, z_{\gamma(G)+1},z_{\gamma(G)+2}, \dots, z_{\tfrac{n+t}{2}} \right\}.$$ For an arbitrary vertex $q \in V(Shu^t_n(G)) \setminus Z$, there exists $1 \leq k \leq \tfrac{n+t}{2}$ such that $q \in A_k$.
        \begin{enumerate}
            \item If $1 \leq k \leq \gamma(G)$, then $q=x_k$ for some $x \in (V(G) \cup \{z\}) \setminus \{a_k\}$. The vertex $a_{k_k} \in Z$ is adjacent to $q$ in $Shu^t_n(G)$. 
            \item If $\gamma(G)+1 \leq k \leq t$, then $q=x_k$ for some $x \in V(G)$. The vertex $z_{k} \in Z$ is adjacent to $q$ in $Shu^t_n(G)$. 
            \item If $t+1 \leq k \leq \tfrac{n+t}{2}$, then $q=x_k$ or $q=y_{n+t+1-k}$ for some $x \in V(G)$ or $y \in V(G) \cup \{z\}$. Case $q=x_k$, where $x \in V(G)$. Since $X$ is a dominating set in $G$, there exists $a_s \in X$ such that $a_sx \in E(G)$. Consequently, the vertex $a_{s_s} \in Z$ is adjacent to $q$ in $Shu^t_n(G)$. Case $q=y_{n+t+1-k}$, where $y \in V(G) \cup \{z\}$. The vertex $z_k \in Z$ is adjacent to $q$ in $Shu^t_n(G)$.
        \end{enumerate}
        Thus, $Z$ is a dominating set in $Shu^t_n(G)$, and $|Z|=\tfrac{n+t}{2}$. Moreover, $\gamma(Shu^t_n(G)) = \tfrac{n+t}{2}$.
        \item Case $\gamma(G) > t$. Define $$Z=\left\{a_{1_1}, a_{2_2}, \dots, a_{t_t}, a_{{(t+1)}_t},\dots, a_{\gamma(G)_t}, z_{t+1},z_{t+2}, \dots, z_{\tfrac{n+t}{2}} \right\}.$$ For an arbitrary vertex $q \in V(Shu^t_n(G)) \setminus Z$, there exists $1 \leq k \leq \tfrac{n+t}{2}$ such that $q \in A_k$.
        \begin{enumerate}
            \item If $1 \leq k \leq t$, then $q=x_k$ for some $x \in (V(G) \cup \{z\}) \setminus \{a_k\}$. The vertex $a_{k_k} \in Z$ is adjacent to $q$ in $Shu^t_n(G)$. 
            \item If $t+1 \leq k \leq \tfrac{n+t}{2}$, then $q=x_k$ or $q=y_{n+t+1-k}$ for some $x \in V(G)$ or $y \in V(G) \cup \{z\}$. Case $q=x_k$, where $x \in V(G)$. Since $X$ is a dominating set in $G$, there exists $a_s \in X$ such that $a_sx \in E(G)$. If $s<t$, the vertex $a_{s_s} \in Z$ is adjacent to $q$ in $Shu^t_n(G)$. If $s \geq t$, the vertex $a_{s_t} \in Z$ is adjacent to $q$ in $Shu^t_n(G)$. Case $q=y_{n+t+1-k}$, where $y \in V(G) \cup \{z\}$. The vertex $z_k \in Z$ is adjacent to $q$ in $Shu^t_n(G)$.
        \end{enumerate}
        Thus, $Z$ is a dominating set in $Shu^t_n(G)$, and $|Z|=\gamma(G)+\tfrac{n-t}{2}$. Moreover, $\tfrac{n+t}{2} \leq \gamma(Shu^t_n(G)) \leq \gamma(G)+\tfrac{n-t}{2}$.
    \end{enumerate}
\end{proof}

We focus on degree-based topological indices of shuriken graphs. Since these indices depend directly on vertex degrees, a precise description of how degrees transform under the shuriken construction is required. The following theorem characterizes the degrees of vertices in the resulting shuriken graph.
\begin{theorem}\label{degree}
    Given graph $G=(V(G),E(G))$ and shuriken graph $Shu^t_n(G)$ of $G$, for some $t,n \in \mathbb{Z}^+$, where $0\leq n-t$ is even. For any vertex $x \in V(G)$,
    $$d_{Shu^t_n(G)}(x_i)=\begin{cases}
        d_G(x)(n-1)+|V(G)|, &\text{if } 1 \leq i \leq t\\
        d_G(x)(n-1)+|V(G)|+1, &\text{if } t+1 \leq i \leq n,
    \end{cases}$$
    and
    $$d_{Shu^t_n(G)}(z_i)=\begin{cases}
        |V(G)|, &\text{if } 1 \leq i \leq t\\
        |V(G)|+1, &\text{if } t+1 \leq i \leq n.
    \end{cases}$$
\end{theorem}
\begin{proof}
    For any $x \in V(G)$. We will find the degree of $x_i$ for any $1\leq i\leq n$ in the graph $Shu^t_n(G)$. Please take a look at the two cases below.
    \begin{enumerate}
        \item Case $1\leq i \leq t$. We know that $x_i$ is adjacent to $y_i$ for all $y \in V(G) \setminus \{x\}$ and adjacent to $z_i$. Thus, there are $|V(G)|$ edges that are incident to $x_i$. Moreover $x_i$ adjacent to all $a_j$ for any $j \in \{1,2,\dots,n\} \setminus\{i\}$ where $ax \in E(G)$. This means that there are $d_G(x)$ possibilities for the vertex $a$. Consequently, in all there are $d_G(x)(n-1)+|V(G)|$ edges that are incident to $x_i$.
        \item Case $t+1\leq i \leq n$. We know that $x_i$ is adjacent to all $a_j$ for any $j \in \{1,2,\dots,n\} \setminus\{n+1-i\}$ where $ax \in E(G)$. This means that there are $d_G(x)$ possibilities for the vertex $a$. Thus, there are $d_G(x)(n-1)$ edges that are incident to $x_i$. Moreover, $x_i$ is adjacent to all vertices $y_{n+1-i}$ for $y \in V(G)$ and adjacent to $z_{n+1-i}$. Consequently, in all there are $d_G(x)(n-1)+|V(G)|+1$ edges that are incident to $x_i$.
    \end{enumerate}
    Furthermore, for $1 \leq i \leq t$, vertex $z_i$ is only adjacent to $x_i$ for all $x \in V(G)$. Thus, $d_{Shu^t_n}(z_i)=|V(G)|$. For $t+1 \leq i \leq n$, vertex $z_i$ is adjacent to $x_{n+1-i}$ for all $x \in V(G)$ and adjacent to $z_{n+1-i}$. Hence, $d_{Shu^t_n}(z_i)=|V(G)|+1$.
\end{proof}

The first Zagreb index depends only on the degrees of vertices. Since the degrees of all vertices in a shuriken graph can be described directly from the base graph, this index can be studied by relating it to the degree sequence of the base graph. The following theorem presents this relationship.
\begin{theorem}\label{firstzagreb}
    Given graph $G=(V(G),E(G))$ and shuriken graph $Shu^t_n(G)$ of $G$, for some $t,n \in \mathbb{Z}^+$, where $0\leq n-t$ is even. If $M_1(G)$ is the first Zagreb index of the graph $G$, then 
    \begin{align*}
        M_1(Shu^t_n(G))&=n(n-1)^2 M_1(G) + n|V(G)|^3 + (3n-2t)|V(G)|^2 \\
        & \quad + 4(n-t)|V(G)| +(n-t) + 4(n-1)|E(G)|(n|V(G)|+(n-t)).
    \end{align*}
\end{theorem}
\begin{proof}
    Consider
    \begin{align*}
        M_1(Shu^t_n(G)) &= \sum_{v \in V(Shu^t_n(G))} (d_{Shu^t_n(G)}(v))^2\\
        & = \sum_{x \in V(G)} \sum_{i=1}^{n} (d_{Shu^t_n(G)}(x_i))^2 + \sum_{i=1}^n(d_{Shu^t_n(G)}(z_i))^2\\
        &= \sum_{x \in V(G)}\bigg( t (d_G(x)(n-1)+|V(G)|)^2 + (n-t)(d_G(x)(n-1)+|V(G)|+1)^2\bigg) \\
        &\quad \quad + t |V(G)|^2 + (n-t) (|V(G)|+1)^2\\
        &= \sum_{x \in V(G)}\bigg( n(n-1)^2(d_G(x))^2 + 2n(n-1)d_G(x)|V(G)| + t |V(G)|^2 \\
        & \quad\quad\quad\quad\quad+ 2(n-t)(n-1)d_G(x) + (n-t)(|V(G)|+1)^2\bigg)\\
        & \quad \quad + t |V(G)|^2 + (n-t) (|V(G)|+1)^2\\
        &= \sum_{x \in V(G)}\bigg( n(n-1)^2(d_G(x))^2 + 2n(n-1)d_G(x)|V(G)| + 2(n-t)(n-1)d_G(x) \bigg)\\
        & \quad \quad + (|V(G)|+1)(t |V(G)|^2 + (n-t) (|V(G)|+1)^2)\\
        &= n(n-1)^2 \sum_{x \in V(G)} (d_G(x))^2 + \left(2n(n-1)|V(G)|+2(n-t)(n-1)\right) \sum_{x \in V(G)}d_G(x)\\
        & \quad + n|V(G)|^3 + (3n-2t)|V(G)|^2 + 4(n-t)|V(G)| + (n-t)\\
        &= n(n-1)^2 M_1(G) + n|V(G)|^3 + (3n-2t)|V(G)|^2 + 4(n-t)|V(G)| \\
        &\quad + (n-t) + 4(n-1)|E(G)|(n|V(G)|+(n-t)).
    \end{align*}
\end{proof}

The second Zagreb index involves pairs of adjacent vertices and their degrees. In a shuriken graph, the adjacency structure is determined by the shuriken construction, which allows this index to be expressed using the degrees of vertices in the base graph. The next theorem formalizes this result.
\begin{theorem}
    Given graph $G=(V(G),E(G))$, where $|V(G)|=v$ and $|E(G)|=e$, and shuriken graph $Shu^t_n(G)$ of $G$, for some $t,n \in \mathbb{Z}^+$, where $0\leq n-t$ is even. If $m_1=M_1(G)$ and $m_2=M_2(G)$, then
    \begin{align*}
        M_2(Shu^t_n(G))=&2e^2(n-1)^2(n-t) + 2etv(n-1) + 2ev(n-1)(n-t)(v+1) \\
			&+ 4ev(n-1)(v-1) + 2e(n-1)(n-t)(v+1) - e(n-t)(2v+1)(-n+t+1) \\
			&- m_1(n-1)(n-t)(-n+t+1) + tv^3 + t (4e^2 - m_1)(n-1)^2 + v^3(v-1)\\
			&+ \frac{1}{2}v^2(n-t)(v+1)^2 + v(n-t)(v+1)^2 + \frac{1}{2}(n-t)(v+1)^2 \\
			&+ \left( ev^2+m_1v(n-1)+m_2(n-1)^2 \right) (n^2-2nt-n+2t^2).
    \end{align*}
\end{theorem}
\begin{proof}
    We put $S=Shu^t_n(G)$. Consider
    \begin{align*}
        M_2(S)&=\sum_{xy \in E(S)}d_S(x)d_S(y)\\
        &= \sum_{\substack{xy \in E(G)\\ 1 \leq i , j \leq t, \text{ } i \neq j\\ t+1 \leq i,j \leq n, \text{ } \{i,j\} \neq \{k,n+t+1-k\}, \\t+1\leq k \leq n}}{d_S(x_i)d_S(y_j)} + \sum_{\substack{a,b \in V(G), a\neq b\\ 1 \leq i \leq t}}{d_S(a_i)d_S(b_i)} + \sum_{\substack{a \in V(G)\\1 \leq i \leq t}}{d_S(z_i)d_S(a_i)}\\
        & \quad + \sum_{\substack{a \in V(G)\\ t+1 \leq i \leq n}}{d_S(z_{n+t+1-i})d_S(a_i)} + \sum_{\substack{a,b \in V(G)\\ t+1 \leq i \leq \tfrac{n+t}{2}}}{d_S(a_i)d_S(b_{n+t+1-i})} + \sum_{i=t+1}^{\tfrac{n+t}{2}}{d_S(z_i)d_S(z_{n+t+1-i})}\\
        &= \sum_{xy \in E(G)}[(d_G(x)(n-1) + v)(d_G(y)(n-1)+v)(t-1)t \\
        &\hspace{2 cm}+ (d_G(x)(n-1)+v+1)(d_G(y)(n-1)+v+1)(n-t)(n-t-1)]\\
        & \quad + \sum_{\substack{a,b \in V(G)\\ a \neq b}}{(d_G(a)(n-1)+v)(d_G(b)(n-1)+v)t} + \sum_{a \in V(G)}v(d_G(a)(n-1)+v)t\\
        & \quad + \sum_{a \in V(G)}{(v+1)(d_G(a)(n-1)+v+1)(n-t)}\\
        & \quad + \sum_{a,b \in V(G)}{(d_G(a)(n-1)+v+1)(d_G(b)(n-1)+v+1)\left(\frac{n-t}{2}\right)} + \frac{n-t}{2}(v+1)^2\\ %
        &= \sum_{xy \in E(G)}[(d_G(x)d_G(y)(n-1)^2 + (d_G(x)+d_G(y))v(n-1) + v^2)(t(t-1)+(n-t)(n-t-1)) \\
        &\hspace{2 cm}+ ((d_G(x)+d_G(y))(n-1)+2v+1)(n-t)(n-t-1)]\\
        & \quad + \sum_{\substack{a,b \in V(G)\\ a \neq b}}{t \left(d_G(a)d_G(b)(n-1)^2 + (d_G(a)+d_G(b))v(n-1) + v^2 \right)} + \sum_{a \in V(G)}\left( d_G(a)vt(n-1)+v^2t\right)\\
        & \quad + \sum_{a \in V(G)}{\left((v+1)d_G(a)(n-1)(n-t)+(v+1)^2(n-t)\right)} + \frac{n-t}{2}(v+1)^2\\
        & \quad + \sum_{a,b \in V(G)}{\frac{n-t}{2} \left(d_G(a)d_G(b)(n-1)^2 + (d_G(a)+d_G(b))(v+1)(n-1) + (v+1)^2 \right)}
    \end{align*}
    \begin{align*}
        M_2(S)&= \left((n-1)^2\sum_{xy \in E(G)}d_G(x)d_G(y) + v(n-1) \sum_{xy \in E(G)}(d_G(x)+d_G(y)) + v^2e \right)(2t^2+n^2-2nt-n) \\
        &\quad + (n-1)(n-t)(n-t-1) \sum_{xy \in E(G)}(d_G(x)+d_G(y))+ (2v+1)e(n-t)(n-t-1)\\
        & \quad + t(n-1)^2\sum_{\substack{a,b \in V(G)\\ a \neq b}}{d_G(a)d_G(b)} + v(n-1)\sum_{\substack{a,b \in V(G)\\ a \neq b}}{(d_G(a)+d_G(b))} + v(v-1)v^2  \\
        & \quad + vt(n-1)\sum_{a \in V(G)} d_G(a)+ v^3t + \frac{n-t}{2}(v+1)^2\\
        & \quad + (v+1)(n-1)(n-t)\sum_{a \in V(G)}{d_G(a)}+ v(v+1)^2(n-t) \\
        & \quad + \tfrac{n-t}{2}(n-1)^2\sum_{a,b \in V(G)}{d_G(a)d_G(b)} + \tfrac{n-t}{2}(v+1)(n-1)\sum_{a,b \in V(G)}(d_G(a)+d_G(b)) + \tfrac{n-t}{2}v^2(v+1)^2.
    \end{align*}
    To begin with, observe the following.
    \begin{align*}
        \sum_{xy \in E(G)}(d_G(x)d_G(y)) &= M_2(G)=m_2\\
        \sum_{xy \in E(G)}(d_G(x)+d_G(y)) &= \sum_{x \in V(G)}(d_G(x))^2=M_1(G)=m_1\\
        \sum_{a,b \in V(G)} d_G(a)d_G(b) &= \sum_{a \in V(G)} d_G(a) \sum_{b \in V(G)} d_G(b)=(2|E(G)|)^2=4e^2\\
        \sum_{\substack{a,b \in V(G)\\ a \neq b}} d_G(a)d_G(b)&=\sum_{a,b \in V(G)} d_G(a) d_G(b) - \sum_{a \in V(G)} d_G(a)d_G(a)=4e^2-m_1\\
        \sum_{a,b \in V(G)} (d_G(a)+d_G(b))&= \sum_{a \in V(G)} \sum_{b \in V(G)} (d_G(a)+d_G(b))=\sum_{a \in V(G)} \left( v d_G(a) + \sum_{b \in V(G)} d_G(b) \right)\\
        &=v\sum_{a \in V(G)}d_G(a) +  v \sum_{b \in V(G)} d_G(b) = 2(2v|E(G)|)=4ve\\
        \sum_{\substack{a,b \in V(G)\\ a \neq b}} (d_G(a)+d_G(b))&=\sum_{a,b \in V(G)} (d_G(a)+d_G(b))- \sum_{a \in V(G)} d_G(a)+d_G(a) = 4ve - 4e = 4e(v-1).
    \end{align*}
    Consequently, we obtain the following result.
    \begin{align*}
        M_2(S)&= (2t^2+n^2-2nt-n)\left((n-1)^2m_2 + v(n-1)m_1 + v^2e \right) + (n-1)(n-t)(n-t-1)m_1 \\
        &\quad + (2v+1)e(n-t)(n-t-1) + t(n-1)^2(4e^2-m_1) + v(n-1)(4e(v-1)) + v^3(v-1)\\
        & \quad + 2evt(n-1)+ v^3t + \frac{n-t}{2}(v+1)^2 + 2e(v+1)(n-1)(n-t)+ v(v+1)^2(n-t) \\
        & \quad + 4\tfrac{n-t}{2}(n-1)^2e^2 + 4\tfrac{n-t}{2}(v+1)(n-1)ve + \tfrac{n-t}{2}v^2(v+1)^2\\
        &= 2e^2(n-1)^2(n-t) + 2etv(n-1) + 2ev(n-1)(n-t)(v+1) + 4ev(n-1)(v-1) \\
        & \quad + 2e(n-1)(n-t)(v+1) - e(n-t)(2v+1)(-n+t+1) - m_1(n-1)(n-t)(-n+t+1) \\
        & \quad + tv^3 + t (4e^2 - m_1)(n-1)^2 + v^3(v-1)+ \frac{1}{2}v^2(n-t)(v+1)^2 + v(n-t)(v+1)^2 \\
        & \quad + \frac{1}{2}(n-t)(v+1)^2 + \left( ev^2+m_1v(n-1)+m_2(n-1)^2 \right) (n^2-2nt-n+2t^2).
    \end{align*}
\end{proof}


Since a necessary condition for a graph to be Eulerian or Hamiltonian is that it must be connected, all the base graphs under consideration in this study are assumed to be non-trivial, that is, graphs that are not null graphs. Due to the structured construction of shuriken graphs, the existence of a Hamiltonian cycle depends on the interaction between the base graph and the shuriken parameters. This motivates the following result, which establishes conditions under which a shuriken graph is Hamiltonian.
\begin{theorem}
    For any semi-Hamiltonian or Hamiltonian graph $G$, the graph $Shu^t_n(G)$ is a Hamiltonian graph, for all $t,n \in \mathbb{Z}^+$, where $0 \leq n-t$ is even.
\end{theorem}
\begin{proof}
    Let $G=(V(G),E(G))$ and $V(G)=\{v^i:1 \leq i \leq k\}$.
    If the graph $G$ is a semi-Hamiltonian or Hamiltonian graph, then there exists a Hamiltonian path in $G$, let say $v^1-v^2-\dots-v^k$. We have $V(Shu^t_n(G))=\bigcup_{j=1}^n\{z_j, v^i_j: 1 \leq i \leq k\}$. We can construct a Hamiltonian cycle in $Shu^t_n(G)$ as follows: $v^1_1-z_1-v^k_1-v^{k-1}_1- \dots - v^2_1-v^1_2-z_2-v^k_2-v^{k-1}_2- \dots-v^2_2-v^1_3-z_3-v^k_3-v^{k-1}_3-\dots-v^2_3- \dots - v^1_t-z_t-v^k_t-v^{k-1}_t-\dots-v^2_t - v^1_{t+1}-z_{n}-v^k_{t+1}-v^{k-1}_{t+1}-\dots-v^2_{t+1} - v^1_{t+2}-z_{n-1}-v^k_{t+2}-v^{k-1}_{t+2}-\dots-v^2_{t+2} - \dots - v^1_{n}-z_{t+1}-v^k_{n}-v^{k-1}_{n}-\dots-v^2_{n}-v^1_1$. Therefore, it follows that the graph $Shu^t_n(G)$ is Hamiltonian.
\end{proof}

Since Eulerian graphs are characterized by vertex degrees and connectivity, the explicit degree structure of shuriken graphs allows a direct analysis of their Eulerian property. The following theorem presents necessary and sufficient conditions for a shuriken graph to be Eulerian.
\begin{theorem}
    For any graph $G$ and $n,t \in \mathbb{Z}^+$, where $0 \leq n-t$ is even, the graph $Shu^t_n(G)$ is Eulerian if and only if $t=n$, $|V(G)|$ is even, and $G$ is Eulerian or $n$ is odd.
\end{theorem}
\begin{proof}
    Let $G$ be an arbitrary non-null graph and $n,t \in \mathbb{Z}^+$, where $0 \leq n-t$ is even.\\
    ($\Longrightarrow$) If the graph $Shu^t_n(G)$ is Eulerian, then $d_{Shu^t_n(G)}(x)$ is even for all $x \in V(Shu^t_n(G))$. Assume that $t \neq n$. Consequently, we have $z_1,z_{t+1} \in V(Shu^t_n(G))$ such that $d_{Shu^t_n(G)}(z_1)=|V(G)|$ and $d_{Shu^t_n(G)}(z_{t+1})=|V(G)|+1$. Thus, their degrees must include even and odd values. Therefore, the graph $Shu^t_n(G)$ is not Eulerian, which leads to a contradiction. Hence, $t=n$. Furthermore, $|V(G)|$ must be even. Since $d_{Shu^t_n(G)}(x_1)=d_G(x)(n-1)+|V(G)|$ for all $x \in V(G)$, it follows that $d_G(x)(n-1)$ must be even for every $x \in V(G)$. If there exists $x \in V(G)$ such that $d_G(x)$ is odd, then $n$ is odd. Therefore, it follows that $d_G(x)$ must be even for every $x \in V(G)$, or $n$ must be odd. In other words, either $G$ is Eulerian or $n$ is odd.\\
    ($\Longleftarrow$) We have $t=n$ and $|V(G)|$ is even.  For all $y_i \in Shu^t_n(G)$ we have $$d_{Shu^t_n(G)}(y_i)=\begin{cases}
        d_G(y)(n-1)+|V(G)|, &\text{if } y \in V(G)\\
        |V(G)|, &\text{if } y=z.
    \end{cases}$$
    If $G$ is Eulerian, then $d_G(y)$ is even for every $y \in V(G)$ or $n$ is odd, it follows that $d_{Shu^t_n(G)}(y_i)$ is even for all $y_i \in V(Shu^t_n(G))$.
    Based on Euler's Theorem and Hierholzer's Theorem, it follows that the graph $Shu^t_n(G)$ is Eulerian.
\end{proof}

\section{Conclusion}
The shuriken graph operation provides a new perspective on graph constructions and enriches the existing collections of graph operations. In particular, this operation arises naturally from algebraic graphs, namely clean graphs and idempotent graphs, which highlights a meaningful connection between graph theory and ring-theoretic structures.

In this paper, we have established several structural and computational properties of shuriken graphs. Specifically, we determined the clique number and the independence number, obtained bounds for the chromatic number and the domination number, and computed the first and second Zagreb indices of shuriken graphs. Furthermore, we derived sufficent or/and necessary conditions under which shuriken graphs are Hamiltonian or Eulerian.

A key observation throughout this study is that the graph parameters of shuriken graphs depend explicitly on the parameters of their base graphs. As a consequence, the analysis of graph invariants for clean graph can be significantly simplified by reducing the problem to the study of the corresponding idempotent graphs. This reduction provides approach that more efficient to understand algebraically motivated graph structures and suggests potential directions for further investigation of graph operations arising from finite rings.

\end{document}